\documentclass[12pt]{amsart}
\usepackage{amsmath,amsxtra,amssymb,latexsym, amscd,amsthm,amsfonts}
\newtheorem{thm}{Theorem}[section]
\newtheorem{cor}[thm]{Corollary}
\newtheorem{lem}[thm]{Lemma}
\newtheorem{prop}[thm]{Proposition}
\theoremstyle{definition}
\newtheorem{defn}[thm]{Definition}
\theoremstyle{remark}
\newtheorem{rem}[thm]{Remark}
\numberwithin{equation}{section}
\newcommand{\R}{\mathbb R}
\newcommand{\N}{\mathbb N}
\newcommand{\simuleq}[1]{\left\{\begin{aligned}#1\end{aligned}\right.}%
\begin{document}%

\title[]{REGULARITY OF VISCOSITY SOLUTIONS DEFINED BY HOPF-TYPE FORMULA FOR  HAMILTON-JACOBI EQUATIONS }%
\author{NGUYEN HOANG}
\thanks{This research is partially supported by
the  Nafosted, Vietnam, under grant \# 101.02-2013.09 and by the project DHH2013-03-35.
}

\address{Department of Mathematics, College of Education, Hue University, 32 Le Loi, Hue, Vietnam}%
\email{nguyenhoanghue@gmail.com     or: nguyenhoang@hueuni.edu.vn}%


\begin{abstract}
Some properties of characteristic curves in connection with viscosity solutions of Hamilton-Jacobi equations $(H,\sigma)$ defined by Hopf-type formula $u(t,x)=\max_{q\in\R^n}\{
\langle x,q\rangle -\sigma^*(q)-\int_0^tH(\tau,q)d\tau\}$ are
studied. We investigate the points where the function $u(t,x)$ is differentiable, and the strip of the form
$(0,t_0)\times \R^n$ of the domain $\Omega$ where the viscosity solution $u(t,x)$ is continuously differentiable. Moreover, we present the propagation of singularity in forward of $u(t,x).$
\end{abstract}
\maketitle
\section{Introduction}
The notion of viscosity solution introduced by Crandall M.G. and
Lions P.L. in \cite{cl} plays a fundamental role in studying
Hamilton-Jacobi equations as well as the related problems such as
calculus of variation, optimal control, etc. By definition, a
viscosity solution of Hamilton-Jacobi equation is merely a
continuous function $u$ satisfying differential inequalities or $u$
is verified a such solution by $C^1$- test functions. As a result,
the relationship between viscosity solutions and classical solutions
is  a subtle matter. Therefore, many authors pay attention to
studying the regularity of viscosity solution in following meanings:
Under what conditions that the viscosity solution $u$ is locally
Lipschitz or differentiable (may be almost everywhere in the domain
of definition $\Omega$ of $u$); finding subregion $V\subset \Omega$
where $u\in C^1(V)$; investigating behaviour of sets where $u$ is
not differentiable, and so on. A large part of these studies is based on
the representation formulas of solutions where Hopf-Lax-Oleinik and
Hopf formulas are especially concerned.

Consider the Cauchy problem for Hamilton-Jacobi
equations of the form \begin{equation}\label{1.1}\frac{\partial u}{\partial t} + H(t,x,D_x u)=0\, , \,\, (t,x)\in \Omega=(0,T)\times \R^n,
\end{equation}
\begin{equation}\label{1.2}
u(0,x)=\sigma(x)\, , \,\, x\in \R^n.\end{equation}

If the Hamiltonian $H(t,x,p)$ is convex in $p,$ the problem
(\ref{1.1})-(\ref{1.2}) is investigated via method of calculus of variation, and
the representation of viscosity solution of Hamilton-Jacobi
equation by the value function associated to the problem of variation may be
considered as a generalized form of Hopf-Lax-Oleinik formula
\begin{equation}\label{1.3}u(t,x)=\min_{y\in \R^n} \Big\{ \sigma
(y)+tH^*\big (\frac {x-y}{t}\big)\Big \}, \end{equation} where
$H=H(p)$ is convex and superlinear, $\sigma$ is Lipschitz on $\R^n,$ here * denotes the Fenchel conjugate. 
Many results on the regularity of viscosity solutions in the case of
convex Hamiltonians are obtained, see \cite{ac,ac1,bcjs,fl}
especially \cite{cs} and references therein.
\smallskip

If $H$ is nonconvex, Hopf formula for viscosity solution of the problem (1.1)-(1.2) is
\begin{equation}\label{1.4}u(t,x)=\max_{q\in\R^n}\{
\langle x,q\rangle -\sigma^*(q)-tH(q)\}\end{equation} under the
assumptions that $H(t,x,p) = H(p)$ is a continuous  function,
 $\sigma (x)$
is convex and Lipschitz, see \cite{h,be,lr}.
\smallskip

A generalization of formula (\ref{1.4}) called Hopf-type formula is that
\begin{equation}\label {htf}u(t,x)=\max_{q\in\R^n}\{
\langle x,q\rangle -\sigma^*(q)-\int_0^tH(\tau,q)d\tau\}\end{equation} where $H=H(t,p)$ is a continuous and $\sigma$
is convex,  is a locally Lipschitz continuous function satisfying the initial condition and equation (\ref{1.1})
 at almost all points in the domain $\Omega,$  i.e., a {\it Lipschitz solution},
 but it is not a viscosity solution in general, see \cite{lr}. Recently, in \cite{nh2} we prove that (\ref{htf})
 defines a viscosity solution of the problem for a specific class of Hamiltonians $H=H(t,p).$
\smallskip

It is noticed that, in the problems of calculus of variation or optimal control,  Hamiltonians are concerned with the convexity
(or concavity) in the global setting, but in differential games, they are neither convex nor concave in general. In \cite{bafa},  Bardi M. and Faggian  S. presented explicit estimates below and above of the form ``maxmin" and ``minmax'' for the viscosity solutions  where either the Hamiltonians or the initial data are the sum of a convex and a concave function. If these estimates are equal, then a representation formula for the solution is obtained.
\smallskip

In this paper we study properties of characteristics of the Cauchy problem where $H=H(t,p)$ in connection with formula
(\ref{htf}). Then we present some results on the existence of strip of differentiability
of  the solution $u(t,x)$ given by this formula as well as the points at which $u(t,x)$ is not differentiable.
\smallskip

The structure of the paper is as follows. In section 2 we suggest a classification  of characteristic curves  at one
 point of the domain and then study the differentiability properties of Hopf-type formula $u(t,x)$ on these curves.
In section 3, we present the conditions related to characteristics so that $u(t,x)$ defined by (\ref{htf})
 is continuously differentiable on the strip $(0,t_0)\times \R^n.$
 Then we show that the singularities of solution  $u(t,x)$ may propagate forward from $t$-time $t_0$
  to the boundary of the domain.
\smallskip

This paper can be considered as a continuation of \cite{nh1} to the
case where dimension of state variable $n$ is greater than 1. The
results obtained here are new, even for Hamiltonian is independent
of $t.$ Our method is to exploit the relationship between
Hopf-type formula and characteristics based on the set of
maximizers.

\medskip

We use the following notations. For a positive number $T$, denote
$\Omega =(0,T)\times \R^n.$ Let $\, |\, .\, |$ and $\langle
.,.\rangle$ be the Euclidean norm and the scalar product in $\R^n$,
respectively. For a function $u:\ \Omega \to \R,$ we denote by
$D_xu$ the gradient of $u$ with respect to variable $x$, i.e., $D_xu
=(u_{x_1},\dots, u_{x_n}),$ and let $B'(x_0,r)$ be the closed ball centered at $x_0$ with radius $r.$ \vspace*{0.05in}

\section{The differentiability of Hopf-type formula and Characteristics}

We now  consider the  Cauchy problem for Hamilton-Jacobi equation:

 \begin{equation}\label{2.1}\frac{\partial u}{\partial t} + H(t, D_x u)=0\, , \,\,
(t,x)\in \Omega =(0,T)\times \R^n,\end{equation}
\begin{equation}\label{2.2}u(0,x)=\sigma(x)\, ,
\,\, x\in \R^n,\end{equation} where the Hamiltonian $H(t,p)$ is of
class $C([0,T]\times\R^n)$  and $\sigma (x)\in C(\R^n)$ is a convex function.

\medskip

Let $\sigma^*$ be the Fenchel conjugate of $\sigma.$ We denote by $D=  $  dom$\, \sigma^*=\{y\in \R^n\ |\, \sigma^*(y)<+\infty\}$ the effective domain of the convex function $\sigma^*.$

\medskip
We assume a  compatible condition for $H(t,p)$ and $\sigma
(x)$ as follows.

\medskip

(A1): {\it For every $(t_0,x_0)\in
[0,T)\times\R^n $,  there exist positive constants $r$ and $N$
such that
$$\langle x,p\rangle -\, \sigma ^* (p)-\int_0^t H(\tau,p)d\tau <
 \max_{|q|\le N}\{\langle x,q\rangle -\, \sigma ^* (q)-\int_0^t H(\tau,q)d\tau\},$$
whenever $ (t,x)\in [0,T)\times\R^n,\, |t-t_0|+|x-x_0|<r$ and
$| p| >N.$}

\medskip
From now on, we denote
\begin{equation}\label{2.3}u(t,x)=\max_{q\in \R^n}\, \{\langle x,q\rangle -\, \sigma^*
(q)-\int_0^t H(\tau,q)d\tau\}.\end{equation}
and
\begin {equation} \label {2.4}\varphi (t,x,q) =\langle
x,q\rangle -\, \sigma ^* (q)-\int_0^t H(\tau,q)d\tau,\ (t,x)\in \Omega,\ q\in \R^n .\end{equation}

For each $(t,x)\in \Omega,$ let $\ell (t,x)$ be the set of all $p\in
\R^n$ at which the maximum of the function $\varphi (t,x,\cdot)$ is
attained. In virtue of (A1), $\ell(t,x)\ne \emptyset.$

\medskip
{\it Remark.} If $\sigma (x)$ is
convex and Lipschitz on $\R^n$  then dom$\sigma^*$ is bounded. Hence condition (A1) is clearly satisfied, 
thus it can be considered as a
generalization of the hypotheses used earlier, see \cite {h,be}.

\medskip
 We record here a theorem that is necessary for
further presentation.
 \begin{thm}\cite{vht} Assume (A1). Then the function
$u(t,x)$ defined by (\ref{2.3})
 is a locally Lipschitz function satisfying equation (\ref{2.1}) a.e. in $\Omega$ and  $u(0,x)=\sigma(x),\ x\in \R^n.$ Furthermore, $u(t,x)$ is of class $C^1(V)$ in some open $V\subset \Omega$ if and only if for
every $(t,x)\in V,\ \ell(t,x)$ is a singleton.\end{thm}

\begin{rem} If  $\ell (t_0,x_0)=\{p\}$  is a singleton, then all partial derivatives of $u(t,x)$ at $(t_0,x_0)$
 exist and $u_x(t_0,x_0)= p,\ u_t(t_0,x_0)=-H(t_0,p)$ see (\cite{vts}, p. 112). Moreover, we have:
\end{rem}

 \begin{thm} Assume (A1). Let $(t_0,x_0)\in \Omega$ such that $\ell(t_0,x_0)$ is a singleton. Then the function
$u(t,x)$ defined by (\ref{2.3}) is differentiable at $(t_0, x_0)$. \end{thm}

\begin{proof}

By assumption, $\ell(t_0,x_0)=\{p\},$  put $p_t = - H(t_0,p).$  For $(h,k)\in \R\times \R^n$ small enough, let
$$\alpha =\limsup_{(h,k)\to (0,0)}\frac{u(t_0+h,x_0+k)-u(t_0,x_0) -p_t h -\langle p,k\rangle}{\sqrt{h^2+|k|^2}},$$

Then there exists a sequence $(h_m,k_m)_m\to 0$ such that $\lim_{m\to \infty} \Phi_m =\alpha,$ where
$$ \Phi_m = \frac{u(t_0+h_m,x_0+k_m)-u(t_0,x_0) -p_t h_m -\langle p,k_m\rangle}{\sqrt{h_m^2+|k_m|^2}}.$$

For each $m\in \N,$ we choose $p_m\in \ell(t_0+h_m,x_0+k_m)$ then
$$\aligned \Phi_m\ \le  &\ \frac{\varphi (t_0+h_m,x_0+k_m,p_m)-\varphi (t_0,x_0,p_m)-p_th_m -\langle p,k_m\rangle}{\sqrt{h_m^2+|k_m|^2}}\\
\le  &\  \frac{-h_m(p_t+H(\tau_m,p_m))  -\langle p_m-p,k_m\rangle}{\sqrt{h_m^2+|k_m|^2}},\endaligned$$
for some $\tau_m$ lying between  $t_0$ and $t_0+h_m;$ $\varphi(t,x,p)$ is given by (\ref{2.4}).
\smallskip

Taking into account the assumption (A1), it is easy to see that, for
$(h_m, k_m)$ small enough, the sequence $(p_m)_m$ is bounded, then
we can choose a subsequence also denoted by $(p_m)_m$ such that
$p_m\to p_0$ as $m\to \infty.$ Since the set-valued mapping
$(t,x)\mapsto \ell(t,x)$ is upper semicontinuous, see \cite{vht}, then
$p_0\in \ell(t_0,x_0),$ that is $p_0=p.$

Now, letting $m\to \infty$ we have
$$\alpha=\lim_{m\to\infty}\, \Phi_m \le \lim_{m\to\infty} \frac{-h_m(p_t+H(\tau_m,p_m)) -\langle p_m-p,k_m\rangle}{\sqrt{h_m^2+|k_m|^2}} =0.$$

On the other hand, let
$$\beta =\liminf_{(h,k)\to (0,0)}\frac{u(t_0+h,x_0+k)-u(t_0,x_0) -p_t h -\langle p,k\rangle}{\sqrt{h^2+|k|^2}}. $$

We have, for $p\in \ell (t_0,x_0)$
$$\aligned u(t_0+h,x_0+k)-u(t_0,x_0) &\ge \varphi (t_0+h,x_0+k,p)-\varphi (t_0,x_0,p) \\
&\ge -hH(\tau^*,p) +\langle p, k\rangle,\endaligned$$
where $\tau^*$ lies between $t_0$ and $t_0+h.$  Therefore
$$\beta \ge \liminf_{(h,k)\to (0,0)}\frac{-h(-p_t-H(\tau^*,p))}{\sqrt{h^2+|k|^2}} =0.$$

Thus,
$$\lim_{(h,k)\to (0,0)}\frac{u(t_0+h,x_0+k)-u(t_0,x_0) -p_t h -\langle p,k\rangle}{\sqrt{h^2+|k|^2}} =0,$$
which shows that $u(t,x)$ is differentiable at $(t_0,x_0).$

The theorem is then proved.
\end{proof}

\begin{defn} We call the function $u(t,x)$ given by (\ref{2.3}) the {\it Hopf-type formula} for Problem (\ref {2.1})-(\ref {2.2}). A point $(t_0,x_0)\in \Omega$ is said  {\it regular} for $u(t,x)$ if the function is differentiable at this point. Other point is said {\it singular} if at which, $u(t,x)$ is not differentiable.
\end{defn}
Consequently, by Theorem 2.3, we see that $(t_0,x_0) \in \Omega$ is regular if and only if $\ell(t_0,x_0)$ is a singleton.
\medskip

Next, in this section we focus on the study the differentiability of Hopf-type formula $u(t,x)$ on the characteristics. To this aim, let us recall the Cauchy method of characteristics for Problem (\ref{2.1})-(\ref{2.2}).
\medskip

From now on, we suppose additionally that $H(t,p)$ and $\sigma (x)$ are of class $C^1$ as a standing assumption.
\medskip

The characteristic differential equations of Problem
(\ref{2.1})-(\ref{2.2}) is as follows
\begin{equation}\label{2.5}
\dot x=H_p \ ;\qquad \dot v = \ \langle H_p,p\rangle
 - \ H \ ;\qquad \dot p=0,\end{equation}
with initial condition
\begin{equation}\label{2.6} x(0)=y \ ;\qquad v(0)=\sigma(y)\ ;\qquad p(0)=
\sigma _y(y)\ ,\quad y\in \R^n.\end{equation}

Then a characteristic strip of Problem
(\ref{2.1})-(\ref{2.2}) (i.e., a solution of the system of
differential equations (\ref{2.5}) - (\ref{2.6})) is defined by
\begin{equation}\label{2.7}\simuleq{x&=x(t,y)=y+\int_0^t H_p(\tau,\sigma_y(y))d\tau, \\
v&=v(t,y)=\sigma(y)+\int_0^t
\langle H_p(\tau,\sigma_y(y)),\sigma_y(y)\rangle d\tau -\int_0^tH(\tau,\sigma_y(y))d\tau,\\
 p&= p(t,y)\ =\ \sigma_y(y).}\end{equation}

The first component of solution (\ref{2.7}) is called the characteristic curve (briefly, characteristics) emanating from $(0,y)$ i.e., the curve defined by
\begin{equation}\label{2.8}\mathcal C:\ x=x(t,y)=y+\int_0^t H_p(\tau,\sigma_y(y))d\tau,\ t\in [0,T]. \end{equation}
\medskip

Let $(t_0,x_0) \in \Omega. $ Denoted by $\ell^*(t_0,x_0)$ the set of
all $y\in \R^n$ such that there is a characteristic curve emanating
from $(0,y)$ and passing the point $(t_0,x_0).$ We have $\ell (t_0,x_0)
\subset {\sigma}_y(\ell^*(t_0,x_0)),$ see \cite{nh1}. Therefore
$\ell^*(t_0,x_0) \ne \emptyset.$

\begin{prop}\label{dt} Let $(t_0,x_0)\in \Omega.$ Then a characteristic curve passing $(t_0,x_0)$ has form
\begin{equation}\label {2.9} x=x(t,y)=x_0+\int_{t_0}^t H_p(\tau,\sigma_y(y))d\tau,\ t\in [0,T], \end{equation}
for some $y\in \ell^*(t_0,x_0).$
\end{prop}

\begin{proof} Let $\mathcal C: \ x=x(t,y)=y+\int_0^t H_p(\tau,\sigma_y(y))d\tau$ be a characteristic curve passing $(t_0,x_0).$ By definition, $y\in \ell^*(t_0,x_0).$  Then we have

$$x_0=y+\int_0^{t_0} H_p(\tau,\sigma_y(y))d\tau$$
Therefore, $$x=x_0-\int_0^{t_0} H_p(\tau,\sigma_y(y))d\tau +\int_0^t H_p(\tau,\sigma_y(y))d\tau = x_0+\int_{t_0}^t H_p(\tau,\sigma_y(y))d\tau.$$

Conversely, let $\mathcal C_1: x=x(t,y)= x_0+\int_{t_0}^t H_p(\tau,\sigma_y(y))d\tau$ for $y\in \ell^*(t_0,x_0)$ be some curve passing $(t_0,x_0).$ Then  we can rewrite $\mathcal C_1$ as:
\begin{equation}\label{2.10} x=x_0-\int_0^{t_0} H_p(\tau,\sigma_y(y))d\tau +\int_0^t H_p(\tau,\sigma_y(y))d\tau.\end{equation}

On the other hand, let $\mathcal C_2:$
\begin{equation}\label{2.11} x= y +\int_0^t H_p(\tau,\sigma_y(y))d\tau\end{equation}
be a characteristic curve also passing $(t_0,x_0).$ Besides that, both $\mathcal C_1,\ \mathcal C_2$ are integral curves of the ODE $x' =H_p(t,\sigma_y(y)),$ thus they must coincide. This proves the proposition.
\end{proof}

\begin{rem}
Suppose that $\sigma_y(y)=p_0\in \ell(t_0,x_0)$ then $y$ is in the subgradient of convex function $\sigma^*$ at $p_0: y\in \partial \sigma^*(p_0).$ Moreover, from (\ref{2.10}) and (\ref{2.11}), we have $y=x_0-\int_0^{t_0} H_p(\tau,p_0)d\tau.$
\end{rem}

Now, let $\mathcal C $ be a characteristic curve passing $(t_0,x_0)$ that is written as
$$x=x(t,y)=x_0+\int_{t_0}^t H_p(\tau,\sigma_y(y))d\tau$$

We say that the characteristic curve $\mathcal C$ is of the {\it type (I) } at the point $(t_0,x_0) \in \Omega$, if $\sigma_y(y) =p\in \ell(t_0,x_0).$ If $\sigma_y(y)\in \sigma_y(\ell^*(t_0,x_0))\setminus \ell(t_0,x_0)$ then $\mathcal C$ is said to be of {\it type (II) } at this point.
\medskip

In the next, we need an additional condition for the Hamiltonian $H=H(t,p).$

\medskip

(A2): The Hamiltonian $H(t,p)$ is admitted as one of two following forms:
\smallskip

a) $H(t,\cdot) $ is a convex or concave function for all $t\in (0,T).$
\smallskip

b) $H(t,p)=g(t)h(p) +k(t)$ for some functions $g,\ h,\  k$ where $g(t)$ does not change its sign for all $t\in (0,T).$

\begin{rem}
1. In particular, if $H(t,p)=H(p)$ then the condition (A2), b)  is
obviously satisfied.
\smallskip

2. In \cite{nh2} we proved that if assumptions (A1) and (A2) are
satisfied, then the function $u(t,x)$ defined by Hopf-type formula
(\ref{2.3}) is a viscosity solution of Problem (2.1)-(2.2).
Moreover, if $\sigma(x)$ is Lipschitz on $\R^n$ then $u(t,x)$ is a
semiconvex function.
\end{rem}
The following lemma is helpful in studying Fenchel conjugate of $C^1$-convex function.
\begin{lem}
Let $v$ be a convex function and $D={\rm dom}v\subset \R^n.$ Suppose that there exist $p,\,  p_0\in D,\ p\ne p_0$ and $y\in \partial v(p_0)$ such that
$$\langle y,p-p_0\rangle = v(p) -v(p_0).$$

Then for all $z$ in the straight line segment $[p,p_0]$ we have
$$v(z) =\langle y,z\rangle -\langle y,p_0\rangle +v(p_0).$$

Moreover, $y\in \partial v(z)$ for all $z\in [p,p_0].$
\end{lem}
\begin{proof}
For $z=\lambda p +(1-\lambda )p_0\in [p,p_0],\ \lambda \in [0,1],$ we write
$$v(z)\le \lambda v(p)+(1-\lambda)v(p_0) =\lambda (v(p)-v(p_0)) +v(p_0).$$

From the hypotheses, we have
$$v(z)\le  \lambda \langle y,p-p_0\rangle +v(p_0).$$ 
Then
$$v(z)-v(p_0) \le \langle y,\lambda p+(1-\lambda) p_0 -p_0\rangle .$$

On the other hand, since $y\in \partial v(p_0), $ the following holds
$$\langle y,\lambda p+(1-\lambda) p_0 -p_0\rangle  \le  v(z)-v(p_0).$$

Thus $$v(z)=\langle y,z\rangle -\langle y,p_0\rangle +v(p_0).$$

Next, let $z\in [p,p_0].$ For any $x\in D, $ we have
$$\aligned v(x)-v(z)=&\ v(x)- \langle y,z\rangle +\langle y,p_0\rangle -v(p_0)\\
= &\ v(x)-v(p_0) -\langle y, z-p_0\rangle\\
\ge &\ \langle x-p_0,y\rangle -\langle z-p_0, y\rangle \ge  \langle x-z,y\rangle.\endaligned$$

This gives us that $y\in \partial v(z).$
\end{proof}

Now we present  properties of characteristic curves of type (I) at $(t_0,x_0)$ given by the following theorems.

\medskip

\begin{thm}\label{cha1} Assume (A1) and (A2).  Let $(t_0,x_0)\in (0,T) \times \R^n,\ p_0= \sigma_y(y)\in \ell (t_0,x_0)$ and let
\begin{equation}\label{2.12} \mathcal C: x= x(t)= x_0 +\int_{t_0}^t H_p(\tau,p_0)d\tau, \ t\in [0, T], \end{equation}
be a characteristic curve of type (I) at $(t_0,x_0).$  Then for all $(t,x)\in \mathcal C, \ 0\le t\le t_0$ one has $p_0\in \ell (t,x)$ and moreover, $\ell(t,x)\subset \ell(t_0,x_0).$ 
\end{thm}

\begin{proof}
Let $(t_1,x_1)\in \mathcal C,\ 0\le t_1\le t_0.$ Take an arbitrary $p\in \R^n$ and denote 
$$ \eta(t,p) =\varphi (t,x,p)-\varphi (t,x,p_0),\ (t,x)\in \mathcal C, \ t\in [0,t_0],$$
where $\varphi (t,x,p)=\langle x,p\rangle -\sigma^*(p)-\int_0^t H(\tau,p)d\tau.$ Then
\begin{equation}\label{2.13} \eta(t,p)=\langle x(t), p-p_0\rangle -(\sigma^*(p)-\sigma^*(p_0))-\int_0^t(H(\tau, p)-H(\tau, p_0))d\tau\end{equation}
for $(t,x)\in \mathcal C.$
\smallskip

We shall prove that $\eta(t,p)\le 0$  for all $t\in [0,t_0].$
\smallskip

It is obviously that, $\eta(t_0,p)\le 0.$ On the other hand, from (\ref{2.13}) and Remark 2.6, we have
$$\eta(0,p)=\langle y, p-p_0\rangle -(\sigma^*(p)-\sigma^*(p_0)),$$
where $y\in \partial \sigma^*(p_0).$ By a property of subgradient of convex function, we have
\begin{equation}\label{2.14}\eta(0,p)=\langle y, p-p_0\rangle -(\sigma^*(p)-\sigma^*(p_0))\le  0.\end{equation}

As a result, we have $\eta(0,p) \le  0; \ \eta(t_0,p)\le 0.$
\smallskip

Since $x=x(t)=x_0+\int_{t_0}^t H_p(\tau,p_0)d\tau, $ then from (\ref{2.13}) we also have
$$\eta'(t,p)=\langle H_p(t,p_0),p-p_0\rangle -(H(t,p)-H(t,p_0)),\ t\in [0,t_0].$$

 In the sequel, we consider the following cases:
\smallskip

{\it Case 1.}
Assume  $H(t,\cdot)$ is convex. Then
$$\langle H_p(t,p_0),p-p_0\rangle \le H(t,p)-H(t,p_0). $$

Therefore $\eta'(t,p)\le 0, $ for all $t\in [0,t_0].$
\smallskip

Similarly, if $H(t, \cdot)$ is a concave function, we have $\eta'(t,p)\ge 0,$ for all $t\in [0,t_0].$

\smallskip
{\it Case 2.} Assume $H(t,p)=g(t)h(p)+k(t),$ and $g(t)$ does not change its sign in $(0,T).$ Then
$$\aligned \eta'(t,p) =& \langle g(t) h_p(p_0), p-p_0\rangle - g(t)(h(p)-h(p_0))\\
=& \big( \langle h_p(p_0),p-p_0\rangle -(h(p)-h(p_0))\big)g(t) =\lambda g(t), \endaligned$$
where $\lambda = \langle h_p(p_0),p-p_0\rangle -(h(p)-h(p_0))$ is a constant.
Therefore, $\eta'(t,p)$ also does not change its sign on $[0,t_0].$
\smallskip

Combining the two cases above, we have, for all $t\in [0,t_0]:$
\smallskip

(i) If $\eta'(t,p)\ge 0$ then $\eta(t_1,p)\le \eta(t_0,p)\le 0.$
\smallskip

(ii) If $\eta' (t,p)\le 0$ then $\eta(t_1,p)\le \eta (0,p)\le 0.$
\smallskip

Thus we obtain $\varphi (t_1,x_1,p)\le \varphi (t_1,x_1,p_0)$ for all $p\in \R^n.$  Consequently, $p_0\in \ell(t_1,x_1)$ for any $(t_1,x_1)\in \mathcal C, \ t_1\in [0,t_0]$ and the first assertion has been proved.
\smallskip

Now, let $p\notin \ell(t_0,x_0).$ If $\eta'(t,p) \ge 0$ then $\eta (t,p)\le \eta(t_0,p)<0.$ Otherwise, if $\eta'(t,p)\le  0,$  we have
$$\eta (t,p)\le \eta(0,p)=\langle y,p-p_0\rangle -(\sigma^*(p)-\sigma^*(p_0)), \ t\in [0, t_0).$$

Since $p\ne p_0,$ then $ \langle y,p-p_0\rangle -(\sigma^*(p)-\sigma^*(p_0)) <0.$
Actually, if it is false, i.e., $\langle y,p-p_0\rangle = (\sigma^*(p)-\sigma^*(p_0)),$
then applying Lemma 2.8, we see that $[p,p_0]$ is contained in $\mathcal D=\{z\in \textrm{dom} \sigma^*\, | \ \partial \sigma^*(z)\ne \emptyset\}$ and
 $\sigma^*$ is not strictly convex on the set $[p,p_0].$ This is a contradiction, since $\sigma(x)$ is of $C^1(\R^n),$
   then it is essentially strictly convex on $\mathcal D.$ In particular, $\sigma^* $ is stricly convex on $[p,p_0], $ see (\cite{ro}, Thm. 26.3). This implies $\eta (t,p) <0.$

Therefore, in any case, if $p\notin \ell(t_0,x_0)$ then $p\notin \ell(t,x).$ The proof  is then complete.
\end{proof}

If we intensify slightly assumption (A2), then we have a stronger result, formulated in the following theorem.

\begin{thm}\label{cha2} Assume (A1) and (A2). In case that $H(t,\cdot)$ is a concave function, we assume in addition that $H(t,\cdot)$ is strictly concave for a.e. $t$ in $(0,T).$ Let $(t_0,x_0)\in \Omega,\ p_0= \sigma_y(y)\in \ell (t_0,x_0)$ and let $\mathcal C: x=x_0 +\int_{t_0}^t H_p(\tau,p_0)d\tau, 0<t<T,$ be a characteristic curve of type (I) at $(t_0,x_0).$  Then $\ell (t,x) =\{p_0\}$ for all $(t,x)\in \mathcal C, \ 0\le t < t_0.$
\end{thm}

\begin{proof}
We use the notation as in the proof of Theorem \ref{cha1}. Take $p\in \ell(t_1,x_1)$  where $(t_1,x_1)\in \mathcal C$ and $t_1\in [0,t_0).$ We check that $p=p_0.$ Assume contrarily, $p\ne p_0.$  Let $$\eta(t,p) =\varphi (t,x,p) -\varphi(t,x,p_0),\ (t,x)\in \mathcal C,\ t\in [0,t_0],$$ where $\varphi (t,x,p) =\langle x,p\rangle -\sigma^* (p) -\int_0^tH(\tau,p)d\tau.$ 

By Theorem \ref{cha1}, $\ell(t_1,x_1)\subset \ell(t_0,x_0),$ thus $\eta(t_0,p)=\eta(t_1,p)=0$ since $p\in \ell(t_0,x_0).$ This implies that
\begin{equation}\label{ssc}\langle x_0, p-p_0\rangle  -(\sigma^*(p)-\sigma^*(p_0))=\int_0^{t_0}(H(\tau, p)-H(\tau, p_0))d\tau\end{equation}

Substracting both sides by $\langle \int_0^{t_0} H_p(\tau,p_0)d\tau, p-p_0\rangle,$ and noticing that $ y= x_0-\int_0^{t_0}H_p(\tau,p_0)d\tau,$ we get
\begin{equation}\label{sc1}\langle y, p-p_0\rangle  -(\sigma^*(p)-\sigma^*(p_0))=\int_0^{t_0}\Big( H(\tau, p)-H(\tau, p_0)-\langle H_p(\tau, p_0),p-p_0\rangle\Big )d\tau.\end{equation}

As mentioned before, since  $p_0=\sigma_y(y)$ then $y\in \partial\sigma^*(p_0).$ Applying Lemma 2.8 and arguing as in the proof of Theorem \ref{cha1}, we see that, when $p\ne p_0:$
\begin{equation}\label{sc}\langle y,p-p_0\rangle -(\sigma^*(p)-\sigma^*(p_0))<0.\end{equation} Thus,
\begin{equation} \label{cc} \int_0^{t_0}\Big( H(\tau, p)-H(\tau, p_0)-\langle H_p(\tau, p_0),p-p_0\rangle\Big)d\tau <0.\end{equation}
We consider the following cases.

{\it Case 1.}  Assume $H(t,\cdot)$ is a convex function. Then
$$ H(t, p)-H(t, p_0)-\langle H_p(t, p_0),p-p_0\rangle \ge 0,\  \forall t\in [0,t_0].$$ This contradicts to (\ref{cc}).
\smallskip

{\it Case 2.} Assume $H(t,p)=g(t)h(p) +k(t). $ The equality (\ref{sc1}) can be rewritten as follows
\begin{equation}\label{2.17}\langle y, p-p_0\rangle  -(\sigma^*(p)-\sigma^*(p_0))=(\int_0^{t_0}g(\tau)d\tau)(h(p)-h(p_0) -\langle h_p(p_0),p-p_0\rangle).\end{equation}

Taking into account the condition (\ref{sc}) and noticing that
$\int_0^{t_0}g(\tau)d\tau \ne 0, $ then from (\ref{2.17})  we deduce that
$$h(p)-h(p_0) -\langle h_p(p_0),p-p_0\rangle \ne 0$$ and that
$$\eta'(t,p) = \big( \langle h_p(p_0),p-p_0\rangle -(h(p)-h(p_0))\big)g(t) $$
does not change its sign on $[0,t_0]$ since $g(t)$ does not change its sign by assumption.

Using the strict monotone property of function $\eta(t,p)$ on $[t_1,t_0]$ we see that, $0= \eta(t_1,p)<\eta (t_0,p) =0$  or $0= \eta(t_1,p)>\eta (t_0,p)= 0.$ This yields a contradiction. 
\smallskip

{\it Case 3.} Assume $H(t,\cdot)$ is a strictly concave function for a.e. $t\in (0,T).$ Then $\eta'(t,p) >0$ a.e. in $(0,t_0).$ We deduce that $0= \eta (t_1,p)<\eta(t_0,p)= 0.$ This is also a contradiction.

Thus, in any case, we have $p=p_0$ and consequently, $\ell (t,x) =\{p_0\}$ for all $(t,x)\in \mathcal C, \ 0\le t < t_0.$
\end{proof}

We have seen that, if the characteristic curve $\mathcal C$ is of type (I) at $(t_0,x_0)$ then it is of the type (I) at any point $(t,x)\in \mathcal C, \ 0\le t\le t_0.$ Nevertheless, for the characteristic curve of type (II), we have the following which is somewhat different.

\begin{thm} Assume (A1) and (A2).  In addition, suppose that $H,\sigma$ are of class $C^2.$  Take $(t_0,x_0)\in \Omega$ and let $\mathcal C: x =x(t) =x_0 +\int_{t_0}^t H_p(\tau,\sigma_y(y_0))d\tau $ be a characteristic curve of type (II) at $(t_0,x_0).$   Then there exists $\theta \in (0,t_0)$ such that $\mathcal C$ is of type (I) at $(\theta, x(\theta))$ and $\mathcal C$ is of type (II) for all point $(t,x) \in \mathcal C,\ t\in (\theta, t_0].$
\end{thm}

\begin{proof}
Let  $\mathcal C:\ x=x_0+\int^t_{t_0}H_p(\tau, \sigma_y(y_0))d\tau$ be the characteristic curve of type (II) at $(t_0,x_0)$ emanating from $(0,y_0).$ Then $\sigma_y(y_0)\in   \sigma_y(\ell^*(t_0,x_0))\setminus \ell(t_0,x_0).$
\smallskip

By the Cauchy method of characteristics, the function defined by Hopf-type formula $u(t,x)$ coincides with the local $C^2$ solution of Problem (2.1)-(2.2), see \cite{cs,vht}. Then there exists $t_1\in (0, t_0)$ such that $u(t,x)$ is differentiable at any point $(t,x(t)) \in \mathcal C,\ u_x(t,x)=\sigma_y(y_0) $ and $\ell(t,x)=\{\sigma_y(y_0)\}, \ 0\le t\le t_1.$ Let
$$\theta =\sup \{ t_1\in [0, t_0)\ | \ \ell(s,x(s))=\{\sigma_y(y_0)\},\  0\le s\le t_1\}.$$

 Since the multivalued mapping $(t,x)\mapsto \ell(t,x)$ is upper semicontinuous, then we get that $\sigma_y(y_0)\in \ell(\theta,x(\theta)).$ It is obvious that, $\theta <t_0$ since $\sigma_y(y_0)\notin \ell(t_0,x_0)$ and $\mathcal C$ is of type (I) at $(\theta, x(\theta)).$ On the other hand, for $t\in (\theta, t_0],\ \mathcal C$ is of type (II) at $(t,x(t))$ by the definition of $\theta$ and Theorem \ref{cha1}.
\end{proof}

For a locally Lipschitz function, it is promising to use the notion
of sub- and superdifferentials as well as reachable gradients, see
\cite {cs}, e.g.,  to study its differentiability. Inspired by a general result for the correspondence between the set of reachable gradients
$D^*u(t,x)$ and the set of minimizers of the problem of calculus of variation $(CV)_{t,x}$ for convex Hamiltonian $H(t,x,p)$ in $p,$ established in \cite{cs}, Th. 6.4.9, p.167, we use Theorem \ref{cha1} to establish a similar relationship between $\ell(t_0,x_0)$ and the set
of reachable gradients. First, we briefly recall definitions of some
kinds of differential  as follows.

\begin{defn} Let $u=u(t,x): \ \Omega \to \R$ and let $(t_0,x_0)\in \Omega.$ For $(h,k)\in \R\times \R^n$ we denote by
$$  \tau (p,q,h,k) =\frac{u(t_0+h, x_0+k)-u(t_0,x_0)-ph -\langle q,k\rangle}{\sqrt{|h|^2 +|k|^2}},$$
\end{defn}
$$D^+u(t_0,x_0)=\{ (p,q)\in\R^{n+1}\, |\ \limsup_{(h,k)\to (0,0)} \tau (p,q,h,k) \ \le \ 0\} $$
$$D^-u(t_0,x_0)=\{(p,q)\in \R^{n+1}\, | \ \liminf_{(h,k)\to (0,0)} \tau(p,q,h,k)\ \ge \ 0\},$$
here $p\in \R,\ q\in \R^n.$

Then $D^+u(t_0,x_0)$ (resp. $D^-u(t_0,x_0)$) is called the {\it superdifferential} (resp. {\it subdifferential}) of $u(t,x)$ at $(t_0,x_0).$

\smallskip

We also define the set $D^*u(t_0,x_0)$ of {\it reachable gradients} of $u(t,x)$ at $(t_0,x_0)$ as follows:

Let $(p,q)\in \R\times\R^n,$ we say that $(p,q)\in D^*u(t_0,x_0) $ if and only if there exists a sequence $(t_k,x_k)_k\subset \Omega\setminus \{(t_0,x_0)\}$ such that $u(t,x)$ is differentiable at $(t_k,x_k)$ and,
$$(t_k,x_k)\to (t_0,x_0), \ (u_t(t_k,x_k), u_x(t_k,x_k))\to (p,q)\  \text{ as}\  k\to \infty.$$

If $u(t,x)$ is a locally Lipschitz function,  then $D^*u(t,x) \ne
\emptyset$  and it is a compact set, see \cite{cs}, p.54.
\smallskip

Now let $u(t,x)$ be the Hopf-type formula and let $(t_0,x_0)\in \Omega.$ We denote by
$$\mathcal H(t_0,x_0)=\{(-H(t_0,q), q)\ | \ q\in \ell(t_0,x_0)\}.$$
 Then a relationship between $D^*u(t_0,x_0)$ and the set $\ell (t_0,x_0)$ is given by the following theorem.

\begin{thm}
Assume (A1) and (A2). In case that $H(t,\cdot)$ is a concave function, we assume in addition that $H(t,\cdot)$ is strictly concave for a.e. $t$ in $(0,T).$ Let $u(t,x)$ be the viscosity solution of Problem (2.1)-(2.2) defined by Hopf-type formula. Then for all $(t_0,x_0)\in \Omega,$ we have
$$D^*u(t_0,x_0)=\mathcal H(t_0,x_0).$$
\end{thm}

\begin{proof}

Let $(p_0,q_0)$ be an element of $\mathcal H(t_0,x_0),$ then $p_0=-H(t_0,q_0)$ for some $q_0\in \ell(t_0,x_0).$ Let $\mathcal C:  x=x_0+\int_{t_0}^t H_p(\tau,q_0)d\tau$ be the characteristic curve of type (I) at $(t_0,x_0).$ By assumption and Theorem \ref{cha1}, all points $(t,x)\in \mathcal C, \ t\in [0, t_0)$ are regular. Put $t_k=t_0-1/k$  and take $(t_k,x_k)\in \mathcal C, k=1,2,\dots$ we see that $(t_k,x_k) \to (t_0,x_0)$ and 
$$(u_t(t_k,x_k),u_x(t_k,x_k))=(-H(t_k,q_0),q_0)\to (-H(t_0,q_0), q_0)\in D^*u(t_0,x_0)$$ as $k\to \infty.$  Therefore, $\mathcal H(t_0,x_0)\subset D^*u(t_0,x_0).$
\smallskip

On the other hand, let $(p,q)\in D^*u(t_0,x_0)$ and $(t_k,x_k)_k\subset \Omega\setminus \{(t_0,x_0)\}$ such that $u(t,x)$ is differentiable at $(t_k,x_k)$ and,
$$(t_k,x_k)\to (t_0,x_0), \ (u_t(t_k,x_k), u_x(t_k,x_k))\to (p,q)\  \text{ as}\  k\to \infty.$$
 Since $$ (u_t(t_k,x_k), u_x(t_k,x_k))= (-H(t_k,q_k),q_k),\ \text{for}\ q_k\in \ell(t_k,x_k)$$ and multivalued function $\ell(t,x)$ is u.s.c, see \cite{vht}, then letting $k\to \infty,$ we see that $q\in \ell(t_0,x_0)$ and $p=\lim_{k\to \infty} -H(t_k,q_k) =-H(t_0,q).$ Thus $(p,q)\in \mathcal H(t_0,x_0).$ The theorem is then proved.
\end{proof}

\section{Regularity of Hopf-type formula}

In this section we will study the sets of the form $V=(0,t_*)\times \R^n\subset \Omega$ such that $u(t,x)$ is continuously differentiable on them. Next,  under minimum assumption we show that,  if $(t_0,x_0)$ is a singular point of $u(t,x),$ then there exists another singular one $(t,x)$ for $t>t_0$ and $x$ is near to $x_0.$ It is worth noticing that, a comprehensive study of singularities of semiconcave/semiconvex functions is presented in \cite{cs}, thus we may apply those results to our case since if $\sigma$ is Lipschitz, then Hopf-type formua $u(t,x)$ is a semiconvex function, see \cite{nh2}.

\begin{thm} Assume (A1) and (A2). Let $u(t,x)$ be the viscosity solution of Problem (\ref{2.1})-(\ref{2.2}) defined by Hopf-type formula (2.3). Suppose that there exists $t_* \in (0,T)$ such that the mapping: $y\mapsto x(t_*,y)=y+\int_0^{t_*} H_p(\tau, \sigma_y (y))d\tau$ is injective. Then $u(t,x)$ is continuously differentiable in the open strip $(0,t_*)\times \R^n.$
\end{thm}

\begin{proof}

Let $(t_0,x_0)\in (0,t_*)\times \R^n$ and let $\mathcal C:$
$$ x=x_0+\int_{t_0}^t H_p(\tau,p_0)d\tau,$$
where $p_0=\sigma_y(y_0)\in \ell(t_0,x_0)$  be the characteristic curve going through $(t_0,x_0)$ defined as in Proposition \ref{dt}.

Let $(t_*,x_*)$ be the intersection point of $\mathcal C$ and plane $\Delta^{t_*} :\ t=t_*.$ Since the mapping $y\mapsto x(t_*,y)$ is injective and $\ell(t_*,x_*)\ne \emptyset,$ so there is a unique  characteristic curve passing $(t_*,x_*).$ This characteristic curve is exactly $\mathcal C.$ Therefore, we can rewrite $\mathcal C$ as follows:
$$ x=x_*+\int_{t_*}^t H_p(\tau,p_*)d\tau$$
where $p_*\in \ell(t_*,x_*).$

By assumption, $\ell^*(t_*,x_*)$ is a singleton, so is $\ell(t_*,x_*).$
Consequently, by Theorem \ref{cha1}, $\mathcal C$ is of type (I) at $(t,x)$ and $\ell(t,x)
=\{p_*\}$ for all $(t,x)\in \mathcal C, \, 0<t< t_*,$  particularly at $(t_0,x_0)$ and
then, $p_*=p_0.$ Applying Theorem 2.1 we see that $u(t,x)$ is of
class $C^1$ in $(0,t_*)\times \R^n.$
\end{proof}

Note that at some point $(t_0,x_0) \in \Omega$ where $u(t,x)$ is differentiable there may be more than one characteristic curve goes through, that is $\ell^*(t_0,x_0)$ may not be a singleton. Next, we have:

\begin{thm} Assume (A1) and (A2). Moreover, let $\sigma$  be Lipschitz on $\R^n.$ Suppose that $\ell(t_*,x)$ is a singleton for every point of the plane $\Delta^{t_*} = \{(t_*,x)\in \R^{n+1}:\  x\in \R^n\},\ 0<t_*\le T.$ Then the viscosity solution $u(t,x)$ of Problem (\ref{2.1})-(\ref{2.2}) defined
by Hopf-type formula (\ref{2.3}) is continuously differentiable in the open strip
 $(0,t_*) \times \R^n.$\end{thm}

\begin{proof}
By assumption, the function $\sigma(x)$ is convex and Lipschitz  on $\R^n,$ then $D=$
 dom\, $ \sigma^*  = \{ q \in \R^n \, |\ \sigma^*(q) < +\infty \} $
is a bounded (and convex) subset in $\R^n.$ We thus  have $\ell(t,x) \subset D$ for all $(t,x) \in
\Omega.$ 

Let $(t_0,x_0) \in (0,t_*)\times \R^n.$ For each $y\in \R^n,$ we put
$$\Lambda (y)=x_0-\int_{t_*}^{t_0} H_p(\tau, p(y))d\tau,$$
where $p(y)\in \ell(t_*,y)\in D.$ Since the multi-valued function  $y\mapsto \ell(t_*,y)$ is u.s.c, and by the hypothesis, $\ell(t_*,y)=\{p(y)\} $ is a singleton for all $y\in \R^n,$ we deduce that the single-valued function $y\mapsto p(y)$ is continuous. Therefore the function $\Lambda:  \R^n\to\R^n,$ defined by 
$$  y\mapsto \Lambda (y)=x_0-\int_{t_*}^{t_0} H_p(\tau,p(y))d\tau$$ is also continuous on $\R^n.$
\smallskip

Since $p(y)$ is in the bounded set $D$ and $H_p(t,p)$ is continuous, there exists $M>0$ such that
$$|\Lambda (y)-x_0|\le \int_{t_0}^{t_*} |H_p(\tau,p(y)|d\tau \le M.$$
Therefore $\Lambda$ is a continuous function from the closed ball $B'(x_0,M)$ into itself. By Brouwer theorem, $\Lambda $ has a fixed point $x_*\in B'(x_0,M), $ i.e., $\Lambda (x_*)=x_*,$ hence
$$x_0=x_*+\int_{t_*}^{t_0} H_p(\tau, p(x_*))d\tau.$$

In other words, there exists a characteristic curve $\mathcal C$ of the type (I) at $(t_*,x_*)$ described as in Theorem \ref{cha1} passing $(t_0,x_0)$. Since $\ell(t_*,x_*)$ is a singleton, so is $\ell(t_0,x_0).$ Applying Theorem 2.1, we see that $u(t,x)$ is continuously differentiable in $(0,t_*)\times\R^n.$
\end{proof}

We note that the hypotheses of above theorems are equivalent to the
fact that, there is a unique characteristic curve of type (I) at
point $(t_*,x_*),\ x_*\in \R^n$ going through $(t_0,x_0).$ In general,
at some point $(t_0,x_0) \in (0,t_*)\times \R^n$ where $u(t,x)$ is
differentiable there may be two or more characteristic curves of
type (I) or (II) at such point $(t_*,x_*),$ one among them goes through $(t_0,x_0),$ thus
$\ell^*(t_*,x_*)$ may not be a singleton. Even neither is
$\ell(t_*,x_*).$ Nevertheless, we have:
\medskip

\begin{thm} Assume (A1) and (A2). In case that $H(t,\cdot)$ is a concave function, we assume in addition that $H(t,\cdot)$ is strictly concave for a.e. $t$ in $(0,T).$  Let $u(t,x)$ be the viscosity solution of Problem (\ref{2.1})-(\ref{2.2}) defined by Hopf-type formula. Suppose that there exists $t_* \in (0,T)$ such that all characteristic curves passing $(t_*,x),\ x\in \R^n$ are of type (I). Then $u(t,x)$ is continuously differentiable in the open strip $(0,t_*)\times \R^n.$
\end{thm}

\begin{proof}

We argue similarly to the proof of Theorem 3.1. Let $(t_0,x_0)\in (0,t_*)\times \R^n$ and let $\mathcal C:$
$$ x=x_0+\int_{t_0}^t H_p(\tau,p_0)d\tau$$
where $p_0=\sigma_y(y_0)\in \ell(t_0,x_0)$  be the characteristic curve going through $(t_0,x_0)$ defined as in Proposition 2.5.

Let $(t_*,x_*)$ be the intersection point of $\mathcal C$ and plane $\Delta^{t_*} :\ t=t_*.$  Then we have
$$ x_*=x_0+\int_{t_0}^{t_*} H_p(\tau,p_0)d\tau$$
Therefore, we can rewrite $\mathcal C$ as
$$ x=x_*-\int _{t_0}^{t_*}H_p(\tau,p_0)d\tau  + \int_{t_0}^t H_p(\tau,p_0)d\tau = x_*+\int_{t_*}^t H_p(\tau,p_0)d\tau,$$
then $\mathcal C$ is also a characteristic curve passing $(t_*,x_*).$ By assumption, $\mathcal C$ is of type (I) at this point, so all $(t,x)\in \mathcal C, \ 0\le t <t_*$ are regular by Theorem 2.10. Thus, $\ell(t_0,x_0)$ is a singleton. As before, we come to the conclusion of the theorem.
\end{proof}

\smallskip

Next, we study a simple propagation of singularities of viscosity solution $u(t,x)$ of the Cauchy problem (2.1)-(2.2) defined by Hopf-type formula.

\begin{thm} Assume (A1) and (A2). Let $(t_0,x_0)\in \Omega$ be a singular point of Hopf-type formula $u(t,x).$ Then for each $\epsilon >0$ there exists $\delta >0$ such that for any $t_*>t_0,\ |t_*-t_0|\le \delta,$ there exists $x_*\in B'(x_0,\epsilon)$ such that $(t_*,x_*)$ is also a singular point.
\end{thm}

\begin{proof}  We use an idea of the proof of Lemma 6.5.1 in \cite {cs} with an appropriate adjustment. Let $(t_0,x_0)\in \Omega$ and let $\epsilon >0.$ Under assumption (A1), for all $(t,x)\in E= [t_0, T]\times B'(x_0,\epsilon)$ there exist positive numbers $r_{tx}$ and $N_{tx}$ such that for all $(t',x')$ satisfying $|t'-t|+|x'-x| <r_{tx}$ then $\ell(t',x')\subset B'(0,N_{tx}).$ Hence, we can cover the compact set $E$ by a finite number balls  centered at $(t,x)_i)$ with radii $r_{(tx)_i},\ i=1,\dots, k.$ We take  positive number $M =\max \{N_{(tx)_i}, \ i=1,\dots, k\},$ then for all $(t,x) \in E$ we get $\ell(t,x)\subset B'(0, M).$ Now we choose $ \delta \in (0,T-t_0]$ satisfying 
$$\delta\, \sup_{|t-t_0|\le T-t_0, |p|\le M}\, |H_p(t,p)| \le \epsilon$$
and fix a $t_*>t_0$ so that $t_*-t_0\le \delta.$
\smallskip

By contradiction, if every point $(t_*, y)$ where $y\in B'(x_0,\epsilon)$ is regular, then $\ell(t_*,y) = \{p(y)\}$ is a singleton. Since the multi-valued function  $y\mapsto \ell(t_*,y)$ is u.s.c, then $y\mapsto p(y)$ is continuous on $B'(x_0,\epsilon).$ Thus, as in the proof of Theorem 3.2, we see that the function $\R^n \ni y\mapsto \Lambda (y)=x_0-\int_{t_*}^{t_0} H_p(\tau,p(y))d\tau$ is also continuous.
\medskip

Note that, if $y\in B'(x_0,\epsilon)$ then
$$|\Lambda (y)-x|\le \int_{t_0}^{t_*} |H_p(\tau,p(y)|d\tau \le \delta\, \sup_{|t-t_0|\le T-t_0, |p|\le M}\, |H_p(t,p)| \le \epsilon.$$
Therefore $\Lambda$ is a continuous function from the closed ball $B'(x_0,\epsilon)$ into itself. By Brouwer theorem, $\Lambda $ has a fixed point $x_*\in B'(x_0,\epsilon), $ i.e., $\Lambda (x_*)=x_*,$ hence,
$$x_0=x_*+\int_{t_*}^{t_0} H_p(\tau, p(x_*))d\tau.$$

In other words, there exists a characteristic curve $\mathcal C$ of the type (I) at $(t_*,x_*)$ described as in Theorem 2.9 passing $(t_0,x_0)$. Since $\ell(t_*,x_*)$ is a singleton, so is $\ell(t_0,x_0).$ This contradicts to the hypothesis.
\end{proof}

\begin{rem} If  $(t_0,x_0)\in \Omega$ is a singular point for $u(t,x)$ and $\epsilon >0,$ then there exists $\delta_1=\delta >$ such that for any $t\in [t_0,t_0+\delta_0] $ we can pick out $x=x(t)\in B'(x_0,\epsilon)$ so that $(t,x)$ is singular. Put $x_1=x(t_1)$ where $t_1=t_0+\delta_1.$ By induction, we can find $(\delta_k)_k$ and $x_k=x(t_k), \ t_k=t_{k-1}+\delta_k$ so that $(t_k,x_k)$ is singular. Since $\delta_k>0$ is dependent on $(t_k,x_k)$ there are two possibilities:
$$\sum_{k=1}^\infty \delta_k <T  \ \text{ or }\  \sum_{k=1}^\infty \delta_k \ge T.$$

In the first case, the singularities of $u(t,x)$ constructed by this way may not propagate to the boundary $t=T,$ otherwise the singularities of $u(t,x)$ exist at some points $(T, x_*).$ Nevertheless, if we assume $\sigma (x)$ is Lipschitz on $\R^n$ as an additional condition, then the number $\delta >0$ in the proof above, can be chosen independently of $(t_i,x_i),\ i=1,2,\dots$
\end{rem}
We have the following:

\begin{thm}
Assume (A1), (A2). Moreover, let $\sigma (x)$  be a Lipschitz function on $\R^n.$ For each $\epsilon >0$ there exists $\delta >0$ such that if $(t_0,x_0)$ is a singular point for $u(t,x),$ then
for any $t_1\in [t_0,t_0+\delta]$ there exists $x_1\in B'(x_0,\epsilon)$ such that $(t_1,x_1)$ is also a singular point.
\end{thm}

\begin{proof}
Since $\sigma(x)$ is convex and Lipschitz, then $D=$\ dom$\sigma^*$ is bounded. Hence, $D\subset B'(0,M)$  for some positive number $M.$ Choose a fixed number $\delta >0$ such that 
$$\delta\, \sup_{0\le t\le T, |p|\le M}\, |H_p(t,p)| \le \epsilon.$$

We argue similarly to the proof of Theorem 3.4. Let $(t_0,x_0)$ be a singular point for $u(t,x).$ If there is $t_*\in (t_0, t_0+\delta]$ such that $(t_*,y)$ is regular for all $y\in B'(x_0,\epsilon)$ then the mapping
$$ y\mapsto \Lambda (y)=x_0-\int_{t_*}^t H_p(\tau,p(y))d\tau$$
is continuous from $B'(x_0,\epsilon)$ into itself. Thus, the mapping has a fixed point $x_*\in B'(x_0,\epsilon).$ This implies that there is a characteristics $\mathcal C$ of type (I) at $(t_*,x_*)$ passing $(t_0,x_0)$ and so $(t_0,x_0)$ is regular. This is a contradiction.
\end{proof}

\begin {cor} Assume (A1), (A2) and $\sigma(x)$ is Lipschitz on $\R^n.$ If $u(t,x)$ has a singular point $(t_0,x_0)\in \Omega,$ then for any $\epsilon >0$ and $t>t_0,$ we can find another singular point $(t,x)$ such that $|x-x_0|\le m\epsilon,$  for some $m\in \N.$ Therefore the singular points of $u(t,x)$ propagate with respect to $t$ as $t$ tends to $T.$
\end{cor}

\begin{proof}
Arguing as in Remark 3.5, we see that for $\epsilon >0$ and $t_0< t\le T,$ there is $m\in \N$ such that $m\delta <t\le (m+1)\delta,$ where $\delta >0$ is defined as in Theorm 3.6. Let $t_i=i\delta, i=0,\dots, m.$  After $m$ steps, we can take $x_m\in B'(x_{m-1},\epsilon)$ such that $(t,x_m)$ is singular and then
$$|x_m-x_0|\le |x_m-x_{m-1}| +\dots +|x_1-x_0| \le m\epsilon.$$

The proof is thus complete.
\end{proof}

\noindent{\bf Example.}  We consider the following problem
$$\frac{\partial u}{\partial t}-2t\ln(1+u_x^2) =0,\
t>0,\ x\in \R ,$$
$$u(0,x)=\frac{x^2}{2},\ x\in \R.$$

A  viscosity solution defined by Hopf-type formula of this problem is:
$$u(t,x)=\max_{y\in \R}\, \{xy-\frac{y^2}{2}+t^2 \ln (1+y^2)\}$$
Let $\varphi (t,x,y)=xy -\frac {y^2}{2} +t^2\ln (1+y^2),$ then $\varphi_y(t,x,y)=x-y+\frac{2t^2y}{1+y^2}.$

A simple computation shows that at point $(t_0,x_0)= (\sqrt 2, \frac 25),$ we have $ \varphi_y(\sqrt 2,\frac 25,y)=0\  \Leftrightarrow\  y_1=2; \ y_2=\frac{ -4+ \sqrt {11}}5,\  y_3=\frac{ -4- \sqrt {11}}5$ and the function $\varphi (t_0,x_0,y)$ attains its maximum at $y_1=2.$
\smallskip

There are three characteristic curves that go through the point $(\sqrt 2, \frac 25)$ as follows:

$\mathcal C_1: \ x=2-\frac {4t^2} 5,\quad  \text{starting at} \ y=2$ and

$\mathcal C_{i} = y_{i} -\frac{2y_{i}t^2}{1+y_{i}^2},\ i=2,\, 3, \  \text{starting at}\ y_{2}= \frac{ -4+ \sqrt {11}}5,\ y_{3}= \frac{ -4- \sqrt {11}}5 .$

We see that $\mathcal C_1$ is the characteristic curve of type (I) at $(t_0,x_0)=(\sqrt 2, \frac 25)$ and $\mathcal C_{2}, \mathcal C_{3}$ are the characteristic curves of type (II) at this point since $\ell (\sqrt 2, \frac 25)=\{\sigma'(y_1)\} =\{2\}$ and $\sigma_y(y_i)\notin \ell(\sqrt 2, \frac 25),\ i=2,\, 3.$ Therefore, $(\sqrt 2,\frac 25)$ is a regular point of $u(t,x)$ although there are 3 characteristic curves intersecting at this point.

\smallskip
Now let $(t_1, x_1)=(t_1,0)$ and let the characteristics $\mathcal C$ starting $y\in \R$ go through $(t_1,0).$ Then $y$ is a root of equation
$y-\frac{2t_1^2y}{1+y^2}=0.$

If $0\le t_1\le \frac 1{\sqrt 2}$ then $(t_1,0)$ is regular point of $u(t,x)$ and $\mathcal C_1:\ x=0$ is of type (I) at $(t_1,0).$

If  $t_1> \frac 1{\sqrt 2}$ then $(t_1,0)$ is singular, since $\ell(t_1,0) =\{y_2,\, y_3\},$ where $y_2=\sqrt{2t^2_1 -1},\ y_3=-\sqrt{2t^2_1-1}.$ In this case, the characteristic curves $\mathcal C_2$ and $\mathcal C_3$ starting at $y_2$ and $y_3$ are of type (I), and $\mathcal C_1$ is of type (II) at $(t_1, 0).$

Let $t_* = \frac 1{\sqrt 2}.$ We have  $\varphi (\frac 1{\sqrt 2},x,y)=xy -\frac {y^2}{2} +\frac 12\ln (1+y^2),$ then $\varphi'_y(\frac 1{\sqrt 2},x,y)=x-y+\frac{y}{1+y^2}$ and $\varphi''_y(\frac 1{\sqrt 2},x,y)=-y^2\frac{3+y^2}{(1+y^2)^2} <0, \ y\ne 0.$ Therefore $\ell(  \frac 1{\sqrt 2},x)$ is a singleton for all $x\in \R.$ Applying Theorem 3.2, we see that the solution $u(t,x)$ is continuously differentiable on the strip $(0, \frac 1{\sqrt 2})\times \R^n.$
\smallskip

At last, the segment $x=0;\ t\in (\frac 1{\sqrt 2},T]$ is a set of singular points for $u(t,x).$ So the singularities of $u(t,x)$ propagate to the boundary.


\end{document}